\documentclass[10pt,reqno]{amsart}
\usepackage{amsmath,amsopn,amssymb,amsthm}
\usepackage[cp1251]{inputenc}
\usepackage[english]{babel}
\usepackage[pagebackref,breaklinks=true,colorlinks=true,linkcolor=blue,citecolor=blue,urlcolor=blue]{hyperref}
\usepackage{graphicx}%
\usepackage{color}

\newtheorem{theorem}{Theorem}[section]

\newtheorem{lemma}[theorem]{Lemma}

\newtheorem{question}[theorem]{Question}

\renewenvironment{proof}[1][Proof]{\noindent\textbf{#1.} }{\ \rule{0.5em}{0.5em}}

\begin{document}

\title[HKW theory for negatively curved homogeneous Finsler manifolds]{Heintze-Kobayashi-Wolf theory for negatively curved homogeneous Finsler manifolds}
\author{Ming Xu}
\address[Ming Xu] {School of Mathematical Sciences,
Capital Normal University,
Beijing 100048,
P.R. China}
\email{mgmgmgxu@163.com}

\begin{abstract}
In this paper, we generalize the Heintze-Kobayashi-Wolf theory to homogeneous Finsler geometry, by proving two main theorems. First, any connected negatively curved homogeneous Finsler manifold is isometric to a Lie group endowed with a left invariant metric, and that Lie group must be simply connected and solvable. Second, the requirement in Heintze's criterion is necessary and sufficient for
a real solvable Lie algebra to generate a Lie group which admits negatively curved left invariant Finsler metrics.

\vbox{}
\end{abstract}

\maketitle

\section{Introduction}

Complete Riemannian manifold with strictly negative section curvature (we will call it {\it negatively curved} for simplicity) is a hot topic, which has been extensively studied \cite{BGS1985,Ha2017}. In homogeneous geometry, using Lie algebraic data to classify negatively curved homogeneous manifolds is a natural thought. However, unlike dealing with
positively curved ones \cite{WZ2018}, there are too many
smooth coset spaces admitting negative curvature,  so that explicitly classifying them is impossible.

Fortunately, we have the following Heintze-Kobayashi-Wolf theory
({\it HKW theory} in short) as a remedial measure:
\begin{enumerate}
\item
By a result of J.A. Wolf \cite{Wo1964} and its refinement by E. Heintze \cite{He1973}, any connected negatively curved homogeneous manifold is isometric to a Lie group, which is endowed with a left invariant metric. So we only need to discuss those negatively curved solvmanifolds.
\item
By a theorem of S. Kobayashi \cite{Ko1962}, any connected negatively curved homogeneous Riemannian manifold must be simply connected. So the classification for negatively curved solvmanifolds is a Lie algebraic problem.
\item
E. Heintze proved that a real solvable Lie algebra
$\mathfrak{g}$ generates a Lie group $G$ which admits negatively curved left invariant Riemannian metrics if and only if $\dim_{\mathbb{R}}\mathfrak{g}=\dim_\mathbb{R}[\mathfrak{g},\mathfrak{g}]+1$ and there exists $y_0\in\mathfrak{g}$ such that $\mathrm{ad}(y_0):[\mathfrak{g},\mathfrak{g}]\rightarrow[\mathfrak{g},\mathfrak{g}]$ only has
eigenvalues with positive real parts (see \cite{He1974} or Theorem \ref{cited-thm-1} below).
\end{enumerate}
To summarize, for each negatively curved homogeneous Riemannian manifold, the HKW theory provides a relatively simple representative for it.
Later, R. Azencott and E.N. Wilson proved similar results for homogeneous non-positive curvature
\cite{AW1976-1,AW1976-2}.

In recent years, researchers studied negative curvature in Finsler geometry, where the negatively curved property requires the flag curvature to be strictly negative everywhere. For example, Akbar-Zadeh's theorem tells us that any compact or homogeneous Finsler manifold with negative constant curvature must be Riemannian \cite{Ak1988,De2012}. Z. Shen proved that a compact Finsler manifold with negative flag curvature and constant S-curvature must be Riemannian \cite{Sh2005}.
Deng and his coworkers proved some rigidity results in the homogeneous context \cite{DH2009,XD2017}.

It is natural to ask
\begin{question}\label{question-1}
Can the HKW theory be generalized to homogeneous Finsler geometry?
\end{question}
The progresses imply a positive answer to Question \ref{question-1}. For example,
S. Deng and Z. Hou proved that
any connected negatively curved homogeneous Finsler manifold is simply connected \cite{DH2007}.
We proved that the criterion in \cite{He1974} can be generalized to some special Finsler solvmanifolds (see Theorem 1.3 in \cite{XD2017}).

In this paper, we completely answer Question \ref{question-1} by two main theorems.

\begin{theorem}\label{main-thm-1}
Any connected homogeneous Finsler manifold is isometric to a Lie group
endowed with a left invariant Finsler metric, and this Lie group must be simply connected and
solvable.
\end{theorem}

\begin{theorem}\label{main-thm-2}
Let $G$ be a connected simply connected solvable Lie group with $\dim_\mathbb{R}G\geq2$,
and we apply the notations
$\mathfrak{g}=\mathrm{Lie}(G)$, $\mathfrak{l}^0=[\mathfrak{g},\mathfrak{g}]$ and $\mathfrak{l}^1=[\mathfrak{l}^0,\mathfrak{l}^0]$.
Then the following claims are equivalent:
\begin{enumerate}
\item $G$ admits a negatively curved left invariant Finsler metric;
\item $\dim_\mathbb{R}\mathfrak{g}=\dim_\mathbb{R}\mathfrak{l}^0+1$
and there exists
$y_0\in\mathfrak{g}$
such that the real linear endomorphism
$\mathrm{ad}_{\mathfrak{l}^0/\mathfrak{l}^1}({u}_0)$ on $\mathfrak{l}^0/\mathfrak{l}^1$ induced by $\mathrm{ad}(y_0)=[y_0,\cdot]$ only has eigenvalues with positive real parts;
\item $\dim_\mathbb{R}\mathfrak{g}=\dim_\mathbb{R}\mathfrak{l}^0+1$ and there exists $y_0\in\mathfrak{g}$ such that the real linear endomorphism $\mathrm{ad}_{\mathfrak{l}^0}(y_0)=[y_0,\cdot]$ on $\mathfrak{l}^0$ only has eigenvalues with positive real parts.
\end{enumerate}
\end{theorem}

To summarize, the HKW theory can still guide us study the negative curvature problem in homogeneous Finsler geometry.


The proof of Theorem \ref{main-thm-1} is very similar to that for
its analog in Riemannian geometry, which is sketchy in the literatures. To make this paper more self contained, we supply the details.
The proof of Theorem \ref{main-thm-2} is very different from that in \cite{He1974}, because
most calculations there are no longer valid in the Finsler context.
Here we use a homogeneous flag curvature formula (see Theorem \ref{cited-thm-3} or Theorem 4.1 in \cite{XDHH2017}),
which qualitatively indicates where to find
a non-negative flag curvature, and we refine the argument which proves Theorem 1.3 in \cite{XD2017} with a linear submersion and careful algebraic discussion.

This paper is scheduled as follows. In Section 2, we summarize some necessary knowledge in
general and homogeneous Finsler geometry. In Section 3, we prove Theorem \ref{main-thm-1}. In Section 4, we prove Theorem \ref{main-thm-2}.
\section{Preliminaries in general and homogeneous Finsler geometries}
\subsection{Minkowski norm and linear submersion}
\label{subsection-2-1}
A {\it Minkowski norm} on a finite dimensional real vector space $\mathbf{V}$ is a continuous function $F:\mathbf{V}\rightarrow[0,+\infty)$ satisfying \cite{BCS2000}:
\begin{enumerate}
\item Regularity: $F|_{\mathbf{V}\backslash\{0\}}$ is a positive smooth function;
\item Positive 1-homogeneity: $F(\lambda y)=\lambda F(y)$, $\forall \lambda\geq0,y\in\mathbf{V}$;
\item Strong convexity: for each $y\in\mathbf{V}\backslash\{0\}$, the {\it fundamental tensor}
$$g_y(u,v)=\frac12\frac{\partial^2}{\partial s\partial t}|_{s=t=0}F^2(y+su+tv),\quad\forall u,v\in\mathfrak{m},$$
is an inner product on $\mathbf{V}$.
\end{enumerate}

Let $F$ and $\overline{F}$ be the Minkowski norms on $\mathbf{V}$ and $\overline{\mathbf{V}}$
respectively. The surjective real linear map $l:\mathbf{V}\rightarrow\overline{\mathbf{V}}$ is called
 a {\it linear submersion} from $F$ to $\overline{F}$, if \cite{AD2001}
 $$\inf\{F(v)| l(v)=\overline{v}\}=\overline{F}(\overline{v}),\quad\forall \overline{v}\in\overline{\mathbf{V}}.$$
For each $F$ and $l$ as mentioned above, there exists a unique $\overline{F}$ on $\mathbf{V}$ such that $l$ is a submersion between Minkowski norms. We call this $\overline{F}$ the {\it Minkowski norm induced by submersion} from $F$ and $l$. The following lemma is useful in later discussion.

Suppose $l:(\mathbf{V},F)\rightarrow(\overline{\mathbf{V}},\overline{F})$ is a linear submersion between Minkowski norms. Then for each $\overline{u}\in\overline{\mathbf{V}}\backslash\{0\}$, there exists a unique $u\in l^{-1}(\overline{u})$ satisfying $F(u)=\overline{F}(\overline{u})$. Denote by $g_u(\cdot,\cdot)$ and $\overline{g}_{\overline{u}}(\cdot,\cdot)$ the fundamental tensors for $F$
and $\overline{F}$ respectively, then this $u\in l^{-1}(\overline{u})$, which is called the {\it horizonal lifting} of $\overline{u}$, can be alternatively determined by $g_u(u,\ker l)=0$. Further more, we have
the following lemma, which is a reformulation of Proposition 2.2 in \cite{AD2001}.
\begin{lemma} \label{lemma-5}
The linear map $l$ induces a linear isometry from the inner product $g_u(\cdot,\cdot)$ on the $g_u(\cdot,\cdot)$-orthogonal complement of $\ker l$ to the inner product $\overline{g}_{\overline{u}}(\cdot,\cdot)$ on $\overline{\mathbf{V}}$.
\end{lemma}
\subsection{Finsler metric and flag curvature}
A {\it Finsler metric} on a smooth manifold $M$ is a continuous function $F: TM\rightarrow[0,+\infty)$, such that $F|_{TM\backslash0}$ is a positive smooth function
and $F(x,\cdot)=F|_{T_xM}$ for each $x\in M$ is a Minkowski norm on $T_xM$ \cite{Sh2001}.

At any point $x$ in a Finsler manifold $(M,F)$, the {\it flag curvature} for the vector $y\in T_xM\backslash\{0\}$ (i.e., the {\it flag pole}) and the tangent plane $\mathbf{P}=\mathrm{span}^\mathbb{R}\{y,u\}\subset T_xM$ (i.e., the {\it flag})
is defined as
$$K(x,y,\mathbf{P})=\frac{g_y(R_y(u),u)}{g_y(y,y)g_y(u,u)-g_y(y,u)^2},$$
in which $R_y:T_xM\rightarrow T_xM$ is the Riemann curvature operator.
When $F$ is Riemannian, the flag curvature coincides with the sectional curvature, which is irrelevant to the choice of $y$. See \cite{Sh2001} for the details.

\subsection{Homogeneous Finsler manifold and a flag curvature formula}
A Finsler manifold $(M,F)$ is called {\it homogeneous}
if its isometry group $I(M,F)$ acts transitively on
$M$ \cite{De2012}.  Since $I(M,F)$ is a Lie transformation group \cite{DH2002}, we can present the homogeneous manifold
$M$ as $M=G/H$. Here $G$ is a Lie subgroup of $I(M,F)$ which acts transitively on $M$,
and $H$ is the isotropy subgroup at the origin $o=eH\in G/H=M$.
When $M$ is connected, we may require $G$ to be connected as well.

A decomposition
$\mathfrak{g}=\mathfrak{h}+\mathfrak{m}$ with $\mathfrak{g}=\mathrm{Lie}(G)$ and $\mathfrak{h}=\mathrm{Lie}(H)$ is called a {\it reductive} decomposition for $(G/H,F)$ if it is
$\mathrm{Ad}(H)$-invariant (in the Lie algebraic level, it implies $[\mathfrak{h},\mathfrak{m}]\subset\mathfrak{m}$).
The following lemma provides a canonical reductive decomposition.

\begin{lemma}\label{lemma-9}The orthogonal decomposition $\mathfrak{g}=\mathfrak{h}+\mathfrak{m}$ with respect to the Killing form of $\mathfrak{g}$
is a reductive composition, such that the maximal nilpotent ideal of $\mathfrak{g}$ is contained in
$\mathfrak{m}$.
\end{lemma}

The Riemannian analog of Lemma \ref{lemma-9} can be found in \cite{XY2021}. Its proof can be naturally transferred to the Finsler context.

The subspace $\mathfrak{m}$ in a reductive decomposition for $(G/H,F)$ can be canonically identified with the tangent space $T_o(G/H)$, such that the $\mathrm{Ad}(H)$-action on $\mathfrak{m}$ coincides with the isotropic $H$-action on $T_o(G/H)$. Then the $G$-invariant Finsler metric $F$ on $G/H$ can be one-to-one corresponded to its restriction to $T_o(G/H)$, which is
any arbitrary $\mathrm{Ad}(H)$-invariant Minkowski norm on $\mathfrak{m}$.
For simplicity, we still use $F$ and $g_y(\cdot,\cdot)$ to denote this norm and its fundamental tensor respectively. See \cite{De2012} for more detailed discussion in homogeneous Finsler geometry.

By homogeneity, we only need to discuss the curvatures of a homogeneous Finsler manifold $(G/H,F)$  at the origin. The following homogeneous flag curvature formula (see Theorem 4.1 in \cite{XDHH2017})  played an important role when we classified positively curved homogeneous Finsler manifolds \cite{DX2017}.

\begin{theorem}\label{cited-thm-3}
Let $(G/H,F)$ be a homogeneous Finsler manifold with the reductive decomposition
$\mathfrak{g}=\mathfrak{h}+\mathfrak{m}$. Suppose that $u$ and $v$ are a commuting pair of linearly independent vectors in $\mathfrak{m}$ and assume $g_u(u,[u,\mathfrak{m}]_\mathfrak{m})=0$. Then  for  $\mathbf{P}=\mathrm{span}\{u,v\}$,   we have
$$K(o,u,\mathbf{P})=\frac{g_y(U(u,v),U(u,v))}{g_y(u,u)g_y(v,v)-g_y(u,v)^2},$$
where $U(u,v)\in\mathfrak{m}$ is determined by
$$2g_{u}(U(u,v),w)=g_{u}([u,w]_\mathfrak{m},v)+
g_{u}(u,[v,w]_\mathfrak{m}),\quad\forall w\in\mathfrak{m}.$$
\end{theorem}

When a Lie group $G$ is viewed as the homogeneous manifold $G/H=G/\{e\}$, the corresponding
homogeneous Finsler metric is called {\it left invariant}. In this situation, the reductive decomposition is unique, i.e., $\mathfrak{g}=\mathfrak{h}+\mathfrak{m}=0+\mathfrak{g}$, so we have the following immediate corollary of Theorem \ref{cited-thm-3}.

\begin{lemma}\label{cor-1}
Let $F$ be a left invariant Finsler metric on the Lie group $G$. Suppose that there exist a commuting pair of linearly independent vectors $u$ and $v$ in $\mathfrak{g}=\mathrm{Lie}(G)$, which satisfies $g_u(u,[\mathfrak{g},u])=0$, then $(G,F)$ is not negatively curved.
\end{lemma}

See \cite{Hu2015-1,Hu2015-2,Xu2021,Xu2022,Xu2023} for more homogeneous curvature formulae in homogeneous Finsler geometry and homogeneous spray geometry.

\subsection{Negatively curved homogeneous Finsler manifold}

%

Let $(M,F)$ be a connected negatively curved homogeneous Finsler manifold. Then it is complete.
The main theorem in \cite{DH2007} tells us
\begin{theorem}\label{lemma-2}
Any connected negatively curved homogeneous Finsler manifold must be simply connected.
\end{theorem}
By Cartan-Hadamard Theorem \cite{BCS2000}, $M$ is homeomorphic to an Euclidean space, which implies

\begin{lemma}\label{lemma-1}
A connected negatively curved homogeneous Finsler manifold $(M,F)$ can be presented as
$M=G/H$, where $G$ is the connected isometry group $I_0(M,F)$, and $H$ is a maximal
compact subgroup of $G$.
\end{lemma}

The proof of Lemma \ref{lemma-1} is contained in the argument proving Theorem 1.1 in \cite{XD2017}.

When $M$ is a connected solvable Lie group and $F$ is a left invariant Riemannian metric,
E. Heintze proved \cite{He1974}

\begin{theorem}\label{cited-thm-1}
Let $G$ be a connected simply connected Lie group with a solvable Lie algebra
$\mathfrak{g}$, then it
admits a negatively curved left invariant Riemannian metric if and only if
$\dim_\mathbb{R}\mathfrak{g}=\dim_\mathbb{R}[\mathfrak{g},\mathfrak{g}]+1$ and
there exists $y_0\in\mathfrak{g}$ such that  $\mathrm{ad}(y_0)=[y_0,\cdot]:[\mathfrak{g},\mathfrak{g}]\rightarrow
[\mathfrak{g},\mathfrak{g}]$ only has eigenvalues with positive real parts.
\end{theorem}

\section{Proof of Theorem \ref{main-thm-1}}
Suppose that $(M,F)$ is a connected negatively curved homogeneous Finsler manifold,
then Lemma \ref{lemma-1} provides $M=G/H$, where $H$ is a maximal compact subgroup of $G=I_0(M,F)$.

\begin{lemma}\label{lemma-6}
$G$ has a trivial center.
\end{lemma}

\begin{proof} Assume conversely that there exists $\rho\in C(G)$ which is not the identity map.
Then $\rho$ generates an infinite discrete subgroup $\Gamma\subset C(G)$, which acts freely and isometrically on $(M,F)$.
Indeed, each $\rho^i$ is a Clifford-Wolf translation and each $\Gamma$-orbit is contained in a geodesic on $(M,F)$.
On the quotient manifold $\overline{M}=M/\Gamma$, $F$ induces a metric $\overline{F}$, such that the covering map $\pi:M\rightarrow \overline{M}$ is locally isometric everywhere. So $(\overline{M},\overline{F})$ is also negatively curved.
The $G$-action on $\overline{M}$ is transitive and isometric, so $(\overline{M},\overline{F})$
is a homogeneous Finsler manifold. Because $\overline{M}$ is not simply connected, we get a contradiction to
Theorem \ref{lemma-2}.
\end{proof}

Let $\mathfrak{g}=\mathfrak{s}+\mathfrak{r}$ be the Levi decomposition for $\mathfrak{g}=\mathrm{Lie}(G)$, where $\mathfrak{r}$ is the maximal solvable ideal and $\mathfrak{s}$ is a semi simple subalgebra of  $\mathfrak{g}$ respectively. Let $\mathfrak{k}$
be a maximally compactly imbedded subalgebra of $\mathfrak{s}$, i.e., $\mathfrak{k}$ generates a maximal
compact connected subgroup of $\mathrm{Int}(\mathfrak{s})$.

\begin{lemma}\label{lemma-7}
$\mathfrak{k}$ generates a compact connected subgroup $K$ in $G$.
\end{lemma}

Notice that a compactly imbedded subalgebra is compact, but generally speaking, it may not generate a compact subgroup. Here the negative curvature condition is crucial.

\begin{proof}We have a Lie algebra direct sum decomposition $\mathfrak{s}=\oplus_{i=1}^m\mathfrak{s}_i$, in which each $\mathfrak{s}_i$ is a simple ideal.
Correspondingly, $\mathfrak{k}=\oplus_{i=1}^m\mathfrak{k}_i$, where each $\mathfrak{k}_i$ is a maximally compact subalgebra of $\mathfrak{s}_i$. Denote by $K_i$ the connected Lie subgroup that $\mathfrak{k}_i$ generates. Then we have $K=K_1\cdots K_m$. To prove Lemma \ref{lemma-7}, we only need to verify that each $K_i$ is compact. There are two cases to consider.

{\bf Case 1}: $\mathfrak{k}_i$ is semi simple. In this case,  $\mathfrak{k}_i$ is compact and semi simple, so $K_i$ is compact.

{\bf Case 2}: $\mathfrak{k}_i$ is not semi-simple. In this case,
 $\mathfrak{k}_i=\mathfrak{k}'_i\oplus\mathfrak{c}(\mathfrak{k}_i)$, where
$\mathfrak{k}'_i=[\mathfrak{k}_i,\mathfrak{k}_i]$ is semi simple and $\dim_\mathbb{R}\mathfrak{c}(\mathfrak{k}_i)=1$.
Meanwhile, $\mathfrak{s}_i$ is a simple ideal of non-compact type, its compact dual $\mathfrak{s}'_i\subset\mathfrak{g}\otimes\mathbb{C}$ is compact simple and has the same
rank as $\mathfrak{k}_i$. Indeed, $(\mathfrak{s}'_i,\mathfrak{k}_i)$ is an irreducible Hermitian symmetric pair.
Let $\mathfrak{t}_i$ be a Cartan subalgebra of $\mathfrak{k}_i$, then it is also a Cartan subalgebra of $\mathfrak{s}'_i$, and it contains $\mathfrak{c}(\mathfrak{k}_i)$.
The $\mathrm{ad}_{\mathfrak{g}\otimes\mathbb{C}}(\mathfrak{s}'_i)$-action on $\mathfrak{g}\otimes\mathbb{C}$ only has purely imaginary weights in $\sqrt{-1}\mathfrak{t}_i^*$.
In particular, for any $u\in\mathfrak{c}(\mathfrak{k}_i)$, the semi simple complex linear endomorphism
$\mathrm{ad}_{\mathfrak{g}\otimes\mathbb{C}}(u)$ on $\mathfrak{g}\otimes\mathbb{C}$ only has purely imaginary eigenvalues.
Since the root system of $\mathfrak{k}_i$ is a subset in that
of $\mathfrak{s}'_i$, and $\mathfrak{c}(\mathfrak{g})$ is the common kernel for all roots of $\mathfrak{k}_i$, we can find a suitable $u\in\mathfrak{c}(\mathfrak{g})\backslash\{0\}$, such that
all eigenvalues of $\mathrm{ad}(u)$ are contained in $2\mathbb{Z}\pi\sqrt{-1}$.
Then
$\mathrm{Ad}(\exp u)=e^{\mathrm{ad}(u)}$ is the identity map on $\mathfrak{g}$.
Since $G$ is connected, $\exp u\in C(G)$, and it must be $e$ by Lemma \ref{lemma-6}.

To summarize, above argument indicates that
$\mathfrak{c}(\mathfrak{k}_i)$ generates a compact Lie subgroup $S^1$. Meanwhile,
the compact semi simple $\mathfrak{k}'_i=[\mathfrak{k}_i,\mathfrak{k}_i]$ generates a compact connected $K'_i$ in $G$. So $K_i=K'_i S^1$ is compact.  The proof of Lemma \ref{lemma-7} is finished.
\end{proof}

\begin{lemma}\label{lemma-8}
There exists a connected solvable subgroup $G'$ of $G$ acting transitively on $M$.
\end{lemma}

\begin{proof}
By Lemma \ref{lemma-1}, $H$ is a maximal compact subgroup of $G$. Since any compact subgroup of $G$ is contained by a maximal one, and all maximal compact subgroups of $G$ are conjugate to each other (see Theorem 14.1.3 in \cite{HN2012}), we may assume without loss of generality that $H$ contains
the compact subgroup $K$ in Lemma \ref{lemma-7}, i.e., in the Lie algebraic level, $\mathfrak{k}\subset\mathfrak{h}$.

For each $\mathfrak{s}_i$, $1\leq i\leq m$, in the proof of Lemma \ref{lemma-7}, we have an Iwasawa decomposition
$\mathfrak{s}_i=\mathfrak{k}_i+\mathfrak{n}_i+\mathfrak{a}_i$ (when $\mathfrak{s}_i$ is compact, $\mathfrak{k}_i=\mathfrak{s}_i$ and $\mathfrak{n}_i=\mathfrak{a}_i=0$). Then $\mathfrak{g}'=\oplus_{i=1}^m(\mathfrak{n}_i+\mathfrak{a}_i)+\mathfrak{r}$ is a solvable subalgebra of $\mathfrak{g}$. Denote by $G'$ the solvable Lie subgroup $\mathfrak{g}'$ generates in $G$.
Since $\mathfrak{k}\subset\mathfrak{h}$, i.e., $\mathfrak{g}'+\mathfrak{h}=\mathfrak{g}$,
so $G'\cdot o$ is an open submanifold in $G/H$. Since both $(G/H,F)$ and $(G'\cdot o, F|_{G'\cdot o})$ are homogeneous, they are both complete. So we have $G'\cdot o=G/H=M$.
\end{proof}

\begin{lemma}\label{lemma-10}
There exists a connected solvable subgroup $G''$ of $G$, which acts freely and transitively on $M$.
\end{lemma}

\begin{proof}Let $G'$ be the connected solvable
Lie subgroup of $G$ given by Lemma \ref{lemma-8}.
Then $M$ can be presented as $M=G'/H'$. Let $\mathfrak{g}'=\mathrm{Lie}(G')$ and $\mathfrak{h}'=\mathrm{Lie}(H')$. Then
Lemma \ref{lemma-9} tells that
$\mathfrak{h}'\cap[\mathfrak{g}',\mathfrak{g}']=0$.
This observation enables us to find a linear complement $\mathfrak{g}''$ of $\mathfrak{h}'$ in $\mathfrak{g}'$ which contains $[\mathfrak{g}',\mathfrak{g}']$.
Indeed, $\mathfrak{g}''$ is a Lie subalgebra of $\mathfrak{g}'$, and it generates a connected solvable Lie subgroup $G''$ in $G'$.
By similar argument as in the proof of Lemma \ref{lemma-8}, we get $G''\cdot o=M$.
Because $\dim_\mathbb{R}G''=\dim_\mathbb{R}M$, the map $\pi:G''\rightarrow M$, $\pi(g)=g\cdot o$ is a smooth covering map. Since $M$ is connected and simply connected, this covering map $\pi$ is a diffeomorphism. So the $G''$-action on $(M,F)$ is free and transitive, which proves Lemma \ref{lemma-10}.
\end{proof}

Now we finish the proof of Theorem \ref{main-thm-1}.

\begin{proof}[Proof of Theorem \ref{main-thm-1}]
Lemma \ref{lemma-10} provides a connected solvable Lie subgroup $G''\subset I_0(M,F)$. The
map $\pi:G''\rightarrow M$, $\pi(g)=g\cdot o$, is a diffeomorphism. Further more, it is equivariant with respect to left translations on $G''$ and the natural action of $G''\subset I(M,F)$ on $M$. So $\pi^* F$ is a left
invariant metric on $G''$, i.e., $\pi$ is an isometry between $(M,F)$ and $(G'',\pi^*F)$.
The  first
statement in Theorem \ref{main-thm-1} is proved.

Suppose that the connected Lie group $G$ admits a negatively curved left invariant Finsler mtric. By Lemma \ref{lemma-2}, $G$ must be connected. We prove the solvability of $G$ by contradiction.
Assume $G$ is not solvable, then in the Levi decomposition $\mathfrak{g}=\mathfrak{s}+\mathfrak{r}$ for $\mathfrak{g}=\mathrm{Lie}(G)$, the semi simple
subalgebra $\mathfrak{s}$ has a nonzero maximal compactly imbedded subalgebra $\mathfrak{k}$.
By Lemma \ref{lemma-7}, $\mathfrak{k}$ generates a compact subgroup $K$ in $G$, with $\dim K>0$. The left translations of $K$ on $G$ is contained in a maximal compact subgroup of $I_0(G,F)$, so Lemma \ref{lemma-1} indicates that left translations of $K$ fixes some element of $G$. This is an impossible because $K\neq\{e\}$. The  second statement in Theorem \ref{main-thm-1} is proved.
\end{proof}

%

\section{Proof of Theorem \ref{main-thm-2}}

\subsection{Some notations and preparation lemmas}
Throughout this section, we apply the following notations and assumptions. Let
$G$ be a connected simply connected solvable Lie group, then $\mathfrak{g}=\mathrm{Lie}(G)$ is solvable.
We denote by
$$\mathfrak{l}^0=[\mathfrak{g},\mathfrak{g}], \ \mathfrak{l}^1=[\mathfrak{l}^0,\mathfrak{l}^0],\
\cdots,\ \mathfrak{l}^i=[\mathfrak{l}^0,\mathfrak{l}^{i-1}],\cdots$$
the descending sequence of $[\mathfrak{g},\mathfrak{g}]$.
Because $\mathfrak{g}$ is solvable, $\mathfrak{l}^0$ is nilpotent. We may assume it is $k$-step nilpotent, i.e., $\mathfrak{l}^k=0$ and $\mathfrak{c}(\mathfrak{l}^0)\supset\mathfrak{l}^{k-1}\neq0$.
For any $0\leq i<j$ and $y_0\in\mathfrak{g}$, $\mathrm{ad}(y_0)$ induces a real linear endomorphism $\mathrm{ad}_{\mathfrak{l}^i/\mathfrak{l}^j}(y_0)$ on $\mathfrak{l}^i/\mathfrak{l}^j$.
We denote by $\mathrm{pr}_j$ the linear projection from $\mathfrak{l}^0$ to $\mathfrak{l}^0/\mathfrak{l}^j$. Here are some obvious facts:
\begin{eqnarray}
\mathrm{ad}_{\mathfrak{l}^i/\mathfrak{l}^j}(y_0)(\mathrm{pr}_j(v))=
\mathrm{pr}_j([y_0,v]),\
e^{\mathrm{ad}_{\mathfrak{l}^i/\mathfrak{l}^j}(y_0)}(\mathrm{pr}_j(v))=
\mathrm{pr}_j(e^{\mathrm{ad}(y_0)}(v)),\ \forall v\in\mathfrak{l}^i.\label{011}
\end{eqnarray}
For simplicity, we use the same notations for real linear maps to denote their complexifications (i.e., their complex linear extension maps). For example, (\ref{011}) is still valid when we choose $v$ from $\mathfrak{l}^i\otimes\mathbb{C}$ and view the projection image $\mathrm{pr}_j(v)$
as a vector in $(\mathfrak{l}^i/\mathfrak{l}^j)\otimes\mathbb{C}=\mathfrak{l}^i\otimes\mathbb{C}
/\mathfrak{l}^j\otimes\mathbb{C}$.

\begin{lemma} \label{lemma-4}
Let $A$ be a real linear endomorphism on a finite dimensional real vector space $\mathbf{V}$.
Then the following statements are equivalent:
\begin{enumerate}
\item $A$ only has eigenvalues with positive real parts;
\item For each $v\in(\mathbf{V}\otimes\mathbb{C})\backslash\{0\}$, $\lim_{t\rightarrow+\infty}e^{tA}(v)=\infty$;
\item For each $v\in\mathbf{V}\backslash\{0\}$, there exits a sequence $t_n\in\mathbb{R}$ satisfying $\lim_{n\rightarrow\infty}t_n=+\infty$ and $\lim_{n\rightarrow\infty}e^{t_nA}(v)=\infty$.
\end{enumerate}
\end{lemma}


\begin{proof}First, we prove (1)$\Rightarrow$(2).
Assume that $A$ only has eigenvalues with positive real parts. Its complex linear extension on $\mathbf{V}\otimes\mathbb{C}$  shares the same eigenvalues as $A$, which only have positive real parts.
Notice that $\mathbf{V}\otimes\mathbb{C}$ is the linear direct sum of
$$(\mathbf{V}\otimes\mathbb{C})_\lambda=\{ v\in\mathbf{V}\otimes\mathbb{C}| \exists m>>0,\mbox{ s.t. } (\lambda I-A)^m(v)=0\}$$
for all eigenvalues $\lambda$ of $A$, and $e^{tA}$ preserves each $(\mathbf{V}\otimes\mathbb{C})_\lambda$. So
we only need to verify
$\lim_{t\rightarrow+\infty}e^{tA}(v)=\infty$ for each $v\in(\mathbf{V}\times\mathbb{C})\backslash\{0\}$. Using the Jordan form of $A$, we can find  $m\in\mathbb{N}\cup\{0\}$ and $v_0,v_1,\cdots,v_m\in(\mathbf{V}\otimes\mathbb{C})_\lambda$, satisfying $v_0=v$, $v_m\neq0$ and
\begin{eqnarray}\label{004}
e^{tA}(v)=e^{\lambda t}(v_0+tv_1+\cdots+t^m v_m),\quad\forall t\in\mathbb{R}.
\end{eqnarray}
By the assumptions $v_m\neq0$ and $\mathrm{Re}\lambda>0$,
we get $\lim_{t\rightarrow+\infty}e^{tA}(v)=\infty$ immediately. This ends the proof for (1)$\Rightarrow$(2).

Next, (2)$\Rightarrow$(3) is a trivial fact.

Finally, we prove (3)$\Rightarrow$(1).
Assume conversely that $A$ has an eigenvalue $\lambda$ with $\mathrm{Re}\lambda\leq0$.
If $\lambda\in\mathbb{R}$, $A$ has a nonzero eigenvector $v\in\mathbf{V}$ satisfying $A(v)=\lambda v$. Then we have $\lim_{t\rightarrow+\infty}e^{tA}(v)=\lim_{t\rightarrow+\infty}e^{\lambda t} v=0$ or $v$, which is a contradiction to (3). If $\lambda=a+b\sqrt{-1}$ with $a\leq0$ and $b\in\mathbb{R}\backslash\{0\}$, then we can find a linearly independent pair $u,v\in\mathbf{V}$,
such that $A(u)=au+ bv$ and $A(v)=-bu+av$. Then $e^{tA}(u)=e^{at}(\cos(b t)u+\sin(b t)v)$, which is periodic when $a=0$ and converges to $0$ when $a<0$. In each situation, we can get a contradiction to (3).
This ends the proof of Lemma \ref{lemma-4}.
\end{proof}

\begin{lemma}\label{lemma-13}
Let $\mathbf{V}$ be a finite dimensional real or complex vector space and $r$  a positive integer.
Suppose that we have $k_i\in\mathbb{N}\cup\{0\}$, $\xi_i=a_i+b_i\sqrt{-1}\in\mathbb{C}$ with $a_i>0$ and $b_i\in\mathbb{R}$, and $w_i\in\mathbf{V}\backslash\{0\}$, $\forall 1\leq i\leq r$.
Assume that the pairs in $\{(k_i,\xi_i),\forall 1\leq i\leq r\}$ are all distinct.
Then there exists a sequence $t_n\in\mathbb{R}$ satisfying $\lim_{n\rightarrow+\infty}t_n=+\infty$
and $\lim_{n\rightarrow+\infty}f(t_n)=\infty$, where $f(t)$ is the $\mathbf{V}$-valued function
$f(t)=\sum_{i=1}^s t^{k_i}e^{\xi_i t}w_i$.
\end{lemma}

\begin{proof}
Suppose that $\{1,\cdots,s\}$ is the set of all indices $i\in\{1,\cdots,r\}$ which satisfies $$a_i=\max\{a_j,\forall 1\leq j\leq r\}\quad and \quad k_i=\max\{k_j| a_j=\max\{a_k,\forall 1\leq k\leq r\}\}.$$ Because $\{(k_i,\xi_i),\forall 1\leq i\leq r\}$ are all distinct, $\{b_1,\cdots,b_s\}$ are all distinct. Direct calculation shows
\begin{eqnarray}\label{008}
(t^{k_1}e^{\xi_1 t})^{-1}f(t)=w_1+\sum_{i=2}^s
e^{(b_i-b_1)t\sqrt{-1}}w_i+o(1),
\end{eqnarray}
where $o(1)$ is with respect to $t\rightarrow+\infty$.

Now we prove Lemma \ref{lemma-13} by contradiction.
Assume conversely that it is not correct, then $f(t)$ is bounded for $t\in[0,+\infty)$, and the right side of (\ref{008}) converges to $0$ when $t$ goes to $+\infty$. Then we must have $s\geq 2$ and
$$\lim_{t\rightarrow+\infty}\sum_{i=2}^s e^{(b_i-b_1)t\sqrt{-1}}w_i=-w_1\neq0.$$ It implies
\begin{eqnarray}\label{009}
\lim_{C\rightarrow+\infty}\int_0^{C}\sum_{i=2}^s e^{(b_i-b_1)t\sqrt{-1}} w_i dt=\infty.
\end{eqnarray}
However, because each $e^{(b_i-b_1)t\sqrt{-1}}$ has zero integral in its periods, we have the estimate
\begin{eqnarray}\label{010}
||\int_0^C\sum_{i=2}^s e^{(b_i-b_1)t\sqrt{-1}}w_i dt||
\leq \sum_{i=2}^s \frac{2\pi||w_i||}{|b_i-b_1|}<+\infty,
\end{eqnarray}
in which $||\cdot||$ is any arbitrary norm on $\mathbf{V}$.
The contradiction between (\ref{009}) and (\ref{010}) ends the proof of Lemma \ref{lemma-13}.
\end{proof}

\subsection{Proof of (1)$\Rightarrow$(2) in Theorem \ref{main-thm-2}}

Assume that  there exists a negatively curved left invariant Finsler metric $F$ on $G$. The first statement in (2) of Theorem \ref{main-thm-2} has been proved in Proposition 4.1 in \cite{XD2017}. To be self contained, we recall its proof here. Since $\mathfrak{g}$ is solvable, $\dim_\mathbb{R}\mathfrak{l}^0<\dim_\mathbb{R}\mathfrak{g}$. So we can find $u\in\mathfrak{g}\backslash \mathfrak{l}^0 $ satisfying $g_u(u,\mathfrak{l}^0)=0$.
Obviously, we have $\dim_\mathbb{R}[u,\mathfrak{g}]\leq\dim_\mathbb{R}\mathfrak{l}^0
<\dim_\mathbb{R}\mathfrak{g}-1$. To prove the first statement in (2) of Theorem \ref{main-thm-2}, we only need to verify
$\dim_\mathbb{R}[u,\mathfrak{g}]=\dim_\mathbb{R}\mathfrak{g}-1$. Assume conversely it is not true, then
the kernel of $\mathrm{ad}(u):\mathfrak{g}\rightarrow[u,\mathfrak{g}]$ contains a vector $v\in\mathfrak{g}\backslash\mathbb{R}u$, i.e., $u$ and $v$ is a linear independent commuting pair.
By Lemma \ref{cor-1}, $F$ can not be negatively curved. This is a contradiction.

Next, we prove the second statement in (2) of Theorem \ref{main-thm-2}. If $\dim_\mathbb{R}\mathfrak{g}=2$, it can not be Abelian, otherwise it has constant zero curvature. Then we can find a basis $\{e_1,e_2\}$ for $\mathfrak{g}$, such that $[e_1,e_2]=e_2$. Choosing $y_0=e_1$, then the second statement in (2) is proved. In the discussion below, we assume $\dim_\mathbb{R}\mathfrak{g}\geq 3$.

By linear submersion, the projection map $\mathrm{pr}_1:\mathfrak{l}^0\rightarrow\mathfrak{l}^0/\mathfrak{l}^1$ and the Minkowski norm $F|_{\mathfrak{l}^0}$ on $\mathfrak{l}^0$ induces
a Minkowski norm $\overline{F}$ on $\mathfrak{l}^0/\mathfrak{l}^1$. We denote by $\overline{g}_\cdot(\cdot,\cdot)$ the fundamental tensor of $\overline{F}$. The fundamental tensor of $F|_{\mathfrak{l}^0}$ coincides with that of for $F$, i.e., $g_{\cdot}(\cdot,\cdot)$, except that all three inputs must be from $\mathfrak{l}^0$.

\begin{lemma}\label{lemma-11}
Choose any $y_0\in\mathfrak{g}\backslash\mathfrak{l}^0$, we have  $$\overline{g}_{\overline{u}}(\overline{u},\mathrm{ad}_{\mathfrak{l}^0/\mathfrak{l}^1}(y_0)
\overline{u})\neq0,\quad\forall\overline{u}\in(\mathfrak{l}^0/\mathfrak{l}^1)\backslash\{0\}.$$
\end{lemma}

\begin{proof}
We prove Lemma \ref{lemma-11} by contradiction. Assume conversely that
\begin{eqnarray}\label{001}
\overline{g}_{\overline{u}}(\overline{u},\mathrm{ad}_{\mathfrak{l}^0/\mathfrak{l}^1}(y_0)
\overline{u})=0\mbox{ for some }\overline{u}\in(\mathfrak{l}^0/\mathfrak{l}^1)\backslash\{0\}.
\end{eqnarray}
Let $u$ be the horizonal lifting of $\overline{u}$, i.e.,
$u\in\mathfrak{l}^0\backslash\mathfrak{l}^1$ satisfies $\mathrm{pr}_1(u)=\overline{u}$ and $g_u(u,\mathfrak{l}^1)=0$. In $[y_0,u]+\mathfrak{l}^1$, there exists a unique $u'$ satisfying $g_u(u',\mathfrak{l}^1)=0$. So the assumption (\ref{001}) implies
\begin{eqnarray}\label{002}
g_u(u,[y_0,u])=g_u(u,[y_0,u]+\mathfrak{l}^1)=g_u(u,u')=
\overline{g}_{\overline{u}}
(\overline{u},\mathrm{ad}_{\mathfrak{l}^0/\mathfrak{l}^1}(y_0)\overline{u})
=0,
\end{eqnarray}
where we have applied Lemma \ref{lemma-5} to get the third equal.

The condition $g_u(u,\mathfrak{l}^1)=0$ implies
\begin{eqnarray}\label{007}
g_u(u,[\mathfrak{l}^0,u])=g_u(u,\mathfrak{l}^1)=0.
\end{eqnarray}
The first statement of (2), which has been proved, indicates that $\mathfrak{g}=\mathfrak{l}^0+\mathbb{R}y_0$, so (\ref{002}) and (\ref{007}) can be summarized as $g_u(u,[\mathfrak{g},u])=0$. To apply Lemma \ref{cor-1} to get the contradiction to negative curvature, we just need to find $v\in\mathfrak{g}\backslash\mathbb{R}u$ satisfying $[u,v]=0$. When $\mathfrak{l}^0$ is $k$-step nilpotent with $k>1$, we can choose $v$ from $\mathfrak{l}^{k-1}\backslash\{0\}\subset\mathfrak{l}^1$. When $\mathfrak{l}^0$ is $1$-step nilpotent, i.e., $\mathfrak{l}^0$ is Abelian, because $\dim_\mathbb{R}\mathfrak{l}^0=\dim_\mathbb{R}\mathfrak{g}-1\geq 2$,
we can choose $v$ from $\mathfrak{l}^0\backslash \mathbb{R}u$. This ends the proof of  Lemma \ref{lemma-11}.
\end{proof}

By Lemma \ref{lemma-11}, we can achieve $\overline{g}_{\overline{u}}(\overline{u},
\mathrm{ad}_{\mathfrak{l}^0/\mathfrak{l}^1}(y_0)\overline{u})>0$ for some
$\overline{u}\in(\mathfrak{l}^0/\mathfrak{l}^1)\backslash\{0\}$, by a possible replacement of $y_0$ with $-y_0$.   If $\dim_\mathbb{R}\mathfrak{l}^0/\mathfrak{l}^1=1$, $\mathrm{ad}_{\mathfrak{l}^0/\mathfrak{l}^1}(y_0)$ has only one eigenvalue, which is positive, so the second statement in (2) of Theorem \ref{main-thm-2} is proved in this case.

Now we consider the situation that $\dim_\mathbb{R}\mathfrak{l}^0/\mathfrak{l}^1>1$. By the connectedness, we have $$\overline{g}_{\overline{u}}(\overline{u},
\mathrm{ad}_{\mathfrak{l}^0/\mathfrak{l}^1}(y_0)\overline{u})>0,\quad\forall
\overline{u}\in(\mathfrak{l}^0/\mathfrak{l}^1)\backslash\{0\}.$$ Then by the positive 2-homogeneity and the continuity, there is a constant
$$c=\min_{\overline{u}\in(\mathfrak{l}^0/\mathfrak{l}^1)\backslash\{0\}}
\frac{\overline{g}_{\overline{u}}(\overline{u},
\mathrm{ad}_{\mathfrak{l}^0/\mathfrak{l}^1}(y_0)\overline{u})}{
\overline{g}_{\overline{u}}(\overline{u},\overline{u})}>0.$$
Fix any $\overline{u}\in(\mathfrak{l}^0/\mathfrak{l}^1)\backslash\{0\}$,
we set $\overline{u}(t)=e^{t\mathrm{ad}_{\mathfrak{l}^0/\mathfrak{l}^1}(y_0)}(\overline{u})$ and
$f(t)=\frac12\overline{F}(\overline{u}(t))^2$. Then for each $t\in\mathbb{R}$, $\overline{u}(t)\neq0$ and $f(t)$ depend on $t$ smoothly.
Calculation shows
$$\frac{d}{dt}f(t)=\overline{g}_{\overline{u}(t)}(\overline{u}(t),
\mathrm{ad}_{\mathfrak{l}^0/\mathfrak{l}^1}(y_0)(\overline{u}(t)))\geq
c\cdot\overline{g}_{\overline{u}(t)}(\overline{u}(t),\overline{u}(t))=cf(t),$$
$f(t)\geq e^{ct}f(0)>0$ when $t\geq0$. So we have $\lim_{t\rightarrow+\infty}e^{t\mathrm{ad}_{\mathfrak{l}^0
/\mathfrak{l}^1}(y_0)}(\overline{u})=\infty$
for any $\overline{u}\in(\mathfrak{l}^0/\mathfrak{l}^1)\backslash\{0\}$.
By (3)$\Rightarrow$(1) in Lemma \ref{lemma-4}, $\mathrm{ad}_{\mathfrak{l}^0/\mathfrak{l}^1}(y_0)$ only has eigenvalues with
positive real parts. This ends the proof of (1)$\Rightarrow$(2) in Theorem \ref{main-thm-2}.

\subsection{Proof of (2)$\Rightarrow$(3) in Theorem \ref{main-thm-2}}

Let $y_0\in\mathfrak{g}$ be the vector given in (2) of Theorem \ref{main-thm-2}. We will prove that it meets the requirement in the second statement in (3) of Theorem \ref{main-thm-2}.

\begin{lemma}\label{lemma-12}
For each $l\in\mathbb{N}$, $\mathrm{ad}_{\mathfrak{l}^{l-1}/\mathfrak{l}^{l}}(y_0)$ only has eigenvalues with positive real parts.
\end{lemma}
\begin{proof}
When $l=1$, Lemma \ref{lemma-12} is just
the second statement in (2) of Theorem \ref{main-thm-2}.
Now we further assume that for $l\in\mathbb{N}$,
$\mathrm{ad}_{\mathfrak{l}^{l-1}/\mathfrak{l}^{l}}(y_0)$ only has eigenvalues with positive real parts.

Using the Jordan forms of $\mathrm{ad}_{\mathfrak{l}^0/\mathfrak{l}^1}(y_0)$ and
$\mathrm{ad}_{\mathfrak{l}^{l-1}/\mathfrak{l}^l}(y_0)$, and similar argument as in the proof of Lemma \ref{lemma-4}, we can get:
\begin{enumerate}
\item For each $u\in\mathfrak{l}^0\otimes\mathbb{C}$, there exists $p\in\mathbb{N}\cup\{0\}$, $n_i\in\mathbb{N}\cup\{0\}$,
$\lambda_i\in\mathbb{C}$ with $\mathrm{Re}\lambda_i>0$,
$u_i\in\mathfrak{l}^0\otimes\mathbb{C}$, $\forall 1\leq i\leq p$, such that
$e^{t\mathrm{ad}_{\mathfrak{l}^0/\mathfrak{l}^1 }
(y_0)}(\mathrm{pr}{}_1(u))
=\sum_{i=1}^p t^{n_i}e^{\lambda_i t}\mathrm{pr{}_1}(u_i)$, i.e.,
\begin{eqnarray}\label{005}
e^{t\mathrm{ad}{}(y_0)}(u)
=\sum_{i=1}^p t^{n_i}e^{\lambda_i t} u_i\quad \mbox{(mod }\mathfrak{l}^1\otimes\mathbb{C}\mbox{)};
\end{eqnarray}
\item For each $v\in\mathfrak{l}^{l-1}\otimes\mathbb{C}$,
there exists $q\in\mathbb{N}\cup\{0\}$,
$m_j\in\mathbb{N}\cup\{0\}$,
$\mu_j\in\mathbb{C}$ with $\mathrm{Re}\mu_j>0$,
$v_j\in\mathfrak{l}^{l-1}\otimes\mathbb{C}$, $\forall 1\leq j\leq q$, such that
$e^{t\mathrm{ad}_{\mathfrak{l}^{l-1}{}/\mathfrak{l}^l{}}(y_0)}
(\mathrm{pr}{}_l(v))
=\sum_{j=1}^q t^{m_j}e^{\mu_j t}\mathrm{pr}{}_l(v_j)$, i.e.,
\begin{eqnarray}\label{006}
e^{t\mathrm{ad}{}(y_0)}(v)
=\sum_{j=1}^q t^{m_j}e^{\mu_j t} v_j\quad \mbox{(mod }\mathfrak{l}^l\otimes\mathbb{C}\mbox{)}.
\end{eqnarray}
\end{enumerate}

Any vector $w\in\mathfrak{l}^l\otimes\mathbb{C}$, mod $\mathfrak{l}^{l+1}\otimes\mathbb{C}$,
is a complex linear combination of vectors of the form $[u,v]$, with $u\in\mathfrak{l}^0\otimes\mathbb{C}$ and $v\in\mathfrak{l}^{l-1}\otimes\mathbb{C}$. Since $e^{t\mathrm{ad}{}(y_0)}$ is a complex Lie algebra  automorphism for each $t\in\mathbb{R}$,
for $u$ and $v$ in (\ref{005}) and (\ref{006}) respectively, we have
\begin{eqnarray*}
e^{t\mathrm{ad}{}(y_0)}([u,v])
&=&[e^{\mathrm{ad}{}(y_0)}(u),
e^{\mathrm{ad}{}(y_0)}(v)]\\
&=&[\sum_{i=1}^p t^{n_i}e^{\lambda_i t} u_i,
\sum_{j=1}^q t^{m_j}e^{\mu_j t} v_j]\quad (\mbox{mod } \mathfrak{l}^{l+1}\otimes\mathbb{C})\\
&=&\sum_{i=1}^p\sum_{j=1}^q t^{n_i+m_j}e^{(\lambda_i+\mu_j)t}[u_i,u_j]
\quad (\mbox{mod }\mathfrak{l}^{l+1}\otimes\mathbb{C}).
\end{eqnarray*}
Now we assume that $w$ is chosen from
$\mathfrak{l}^l\otimes\mathbb{C}\backslash\mathfrak{l}^{l+1}\otimes\mathbb{C}$, then for each $t\in\mathbb{R}$, $e^{t\mathrm{ad}(y_0)}(w)\in\mathfrak{l}^l\otimes\mathbb{C}
\backslash\mathfrak{l}^{l+1}\otimes\mathbb{C}$. Above calculations and observations provide
an integer $r>0$, $k_i\in\mathbb{N}\cup\{0\}$, $\xi_i\in\mathbb{C}$ with $\mathrm{Re}\xi_i>0$, $w_i\in\mathfrak{l}^l\otimes\mathbb{C}\backslash\mathfrak{l}^{l+1}\otimes\mathbb{C}$,
$\forall 1\leq i\leq r$, such that the pairs in
$\{(k_i,\xi_i),\forall 1\leq i\leq r\}$ are all distinct and
\begin{eqnarray}\label{012}
e^{t\mathrm{ad}{}
(y_0)}(w)=\sum_{i=1}^r t^{k_i}e^{\xi_i t}w_i,\quad \mbox{(mod }\mathfrak{l}^{l+1}\otimes\mathbb{C}\mbox{)}.
\end{eqnarray}
The equality (\ref{012}) can equivalently presented as
\begin{eqnarray*}
e^{t\mathrm{ad}_{\mathfrak{l}^l{}/\mathfrak{l}^{l+1}{}}(y_0)}
(\mathrm{pr}_{l+1}{}(w))=
\sum_{i=1}^r t^{k_i}e^{\xi_i t}\mathrm{pr}_{l+1}(w_i),
\end{eqnarray*}
where $\mathrm{pr}_{l+1}(w_i)$ is nonzero in 
$\mathfrak{l}^l\otimes\mathbb{C}/\mathfrak{l}^{l+1}\otimes\mathbb{C}$, $\forall 1\leq i\leq r$. 
Lemma \ref{lemma-13} provides a sequence $t_n\in\mathbb{R}$ satisfying $$\lim_{n\rightarrow+\infty}t_n=+\infty\quad \mbox{and}\quad \lim_{n\rightarrow+\infty}e^{t\mathrm{ad}_{\mathfrak{l}^l{}/\mathfrak{l}^{l+1}
{}}(y_0)}(\mathrm{pr}_{l+1}(w))=\lim_{n\rightarrow+\infty}\sum_{i=1}^r t_n^{k_i}e^{\xi_i t_n}\mathrm{pr}_{l+1}(w_i)=\infty.$$
By (3)$\Rightarrow$(1) in Lemma \ref{lemma-4},
$\mathrm{ad}_{\mathfrak{l}^l/\mathfrak{l}^{l+1}}(y_0)$ only has eigenvalues with positive real parts. This ends the proof of Lemma \ref{lemma-12} by induction.
\end{proof}

The spectrum (i.e., eigenvalue set, counting multiplicities) of $\mathrm{ad}_{\mathfrak{l}^0}(y_0):\mathfrak{l}^0\rightarrow\mathfrak{l}^0$ is the union of those of $\mathrm{ad}_{\mathfrak{l}^l/\mathfrak{l}^{l+1}}(y_0):\mathfrak{l}^l/\mathfrak{l}^{l+1}
\rightarrow\mathfrak{l}^l/\mathfrak{l}^{l+1}$, $\forall 0\leq l\leq k-1$. So the second statement in (3) of Theorem \ref{lemma-12} follows after Lemma \ref{lemma-12} immediately.

\noindent

Finally, we remark that (3)$\Rightarrow$(1) in Theorem \ref{main-thm-2} is an immediate corollary of Theorem \ref{cited-thm-1}. This ends the proof of Theorem \ref{main-thm-2}.

{\bf Acknowledgement}.
This paper is supported by Beijing Natural Science Foundation (No. 1222003), National Natural Science Foundation of China (No. 12131012, No. 11821101).
\bigskip

\noindent
{\bf Declarations}
\medskip

\noindent
{\bf Data Availability}\quad Not applicable.
\medskip

\noindent
{\bf Conflict of interest}\quad Not applicable.

\end{document}